\documentclass[12pt,notitlepage]{article}%{amsart}%{article}
\usepackage{amsmath}
\usepackage{amssymb}
\usepackage{amscd}
\usepackage{amsfonts}
\usepackage{amsthm}
\usepackage{mathrsfs}
\usepackage[matrix,arrow,curve]{xy}
\usepackage{amsbsy}
\usepackage{slashed}

  \newtheorem{theorem}{Theorem}[section]
  \newtheorem{df}[theorem]{Definition}%[section]
  \newtheorem{prop}[theorem]{Proposition}
  \newtheorem{lemma}[theorem]{Lemma}
  
\theoremstyle{remark}
  \newtheorem{rem}[theorem]{Remark}

\newcommand{\Dj}{\hbox to 8pt{\raisebox{.4\height}{-}\hss D}}

\newcommand{\ac}{\ensuremath{\mathcal A}}

\newcommand{\lc}{\ensuremath{\mathcal L}}
\newcommand{\uc}{\ensuremath{\mathcal U}}

\newcommand{\g}{\ensuremath{\mathfrak{g}}}

\newcommand\beq{\begin{equation}}
\newcommand\eeq{\end{equation}}
\newcommand\bqa {\begin{eqnarray}}
\newcommand\eqa {\end{eqnarray}}

\newcommand{\bear}{\begin{array}}
\newcommand{\enar}{\end{array}}

\newcommand{\R}{\mathbb{R}}
\newcommand{\C}{\mathbb{C}}

\usepackage[usenames]{color}

\usepackage{graphicx}
\usepackage{wrapfig}

\begin{document}

\begin{center}
\textsc{$L_\infty$-derivations and the argument shift method for deformation quantization algebras}

\medskip
G.Sharygin
\end{center}

\begin{abstract}
The argument shift method is a well-known method for generating commutative families of functions in Poisson algebras from central elements and a vector field, verifying a special condition with respect to the Poisson bracket. In this notice we give an analogous construction, which gives one a way to create commutative subalgebras of a deformed algebra from its center (which is as it is well known describable in the terms of the center of the Poisson algebra) and an $L_\infty$-differentiation of the algebra of Hochschild cochains, verifying some additional conditions with respect to the Poisson structure.
\end{abstract}

\bigskip
\section{Introduction: classical argument shift method}
\subsection{History and motivation}
In the study of integrable systems one naturally and inevitably encounters the question, whether there exists a sufficiently large set of first integrals of the given Hamiltonian equation. To answer it, it is often convenient to have a large collection of commutative subalgebras in the Poisson algebra of functions on the phase space. Thus, constructing and classifying such algebras is an important indispensable part of the integrable systems theory.

Among other methods of constructing commutative families of functions, the argument shift method is one of the simplest and relatively universal. It was first observed in the papers by Manakov in a particular case of Euler equation (see \cite{Manakov1976}), and later it was formulated in full generality (in the case, where the Poisson manifold is equal to the coadjoint representation of a Lie algebra equipped with the standard Poisson structure) by Mischenko and Fomenko, \cite{MiFo78}.

Since that time, the method has been the subject of minute discussions and numerous generalizations (one of which we explain in this paper). In particular, it was shown that under mild conditions the commutative algebras yielded by it are maximal (if the direction of the shift is accurately chosen) and complete; see the papers of Bolsinov, Sadetov, Zhang, Izosimov and others, \cite{Bol1,Bol2,Sad,BolZh,Izo}. 

On the other hand, according to Kontsevich (see \cite{Kon97}) and many others one can apply a quantization procedure to any Poisson manifold, thus obtaining the ``quantum observable'' algebra: the associative noncommutative algebra, linearly isomorphic to the space of (usually smooth, or polynomial) functions on the phase space (often with a formal parameter $\hbar$ added to the picture), such that the product in it is a deformation of the usual commutative product of functions, and the linear part of the deformation is given by the Poisson structure (see the discussion at the end of the section \ref{sec:secintro2} below). One can regard this algebra as the suitable domain for the investigation of the quantum mechanical  problems; similarly to the classical Hamiltonian mechanics, solving such system involves finding a suitable system of mutually commuting elements in the quantized algebra (this time the commutation is understood in the usual algebraic sense). It is natural to assume that such algebras should be somehow related with the commutative Poisson subalgebras on the same space.

This idea, however simple in seems at first sight, is pretty hard to put into practice; the search for the corresponding commutative quantum algebras involves many nontrivial constructions and derives inspiration from most variagated sources. For example, in case when the phase space is given by the dual space of a Lie algebra, this involves the study of universal enveloping algebras and their generalizations, like Yangians and quantum groups, see \cite{FeiginFrenkel,Molev1}.

In this paper we are discussing a possible way to construct quantum counterpart of commutative algebras, yielded by the argument shift method. It turns out that in spite of being very algebraic in its nature, and allowing numerous interpretations on classical level, this method does not allow a straightforward interpretation in quantum case. The known constructions, which give ``quantum integrable systems'', related to this method, involve hard results about the structure of universal enveloping algebras, Yangians, or the properties of the universal enveloping algebras of affine Lie algebras at the critical level, see \cite{Molev2,Molev3,Tarasov,Rybnikov,Talalaev}.

In this paper we suggest an algebraic construction, which generalizes the argument shift method to the deformed algebra. It is based on the theory of $L_\infty$-algebras and $L_\infty$-morphisms. One can say, that it allows us to obtain an analogue of the procedure, that generates the commutative subalgebras, rather than constructs the commutative subalgebras in the quantized algebra directly from the algebras, generated by the shift at the classical level. A considerable drawback of this construction is that to this moment I have no example, where this procedure actually works, the conditions, that should be satisfied for it to be defined, being too hard to observe. I hope to amend this in the papers to come.

 The remaining part of this paper is organised as follows: in section \ref{sec:secintro2} we recall the classical version of the argument shift method, which we phrase in a greater generality than it is usually done: it turns out, that the method is based on a purely algebraic consideration and does not depend on the actual geometric nature of the ingredients, it involves. In section 2 we recall the definitions and basic properties of $L_\infty$ algebras and $L_\infty$-morphisms. We also recall the role they play in Kontsevich's deformation quantization construction. Then in section 3 we define $L_\infty$ derivations of a DG Lie algebra and give the definition of weak Nijenhuis property. We then use this notion in the particular case of the Lie algebra of local Hochschild cochains (polydifferential operators) to show how one can get the first nontrivial commutation relation, analogous to $\{\xi(f),\xi(g)\}=0$ in the proposition \ref{prop:ashiftcl}. In order to move forward we need to replace the weak Nijenhuis condition by the strong Nijenhuis condition, which we do in section 4; then we show (theorem \ref{theo:om}), how in this case the $L_\infty$-derivation gives rise to the analogue of the argument shift method. Finally, in the last section we make few remarks on the possible further developments and applications of our ideas.
 
\medskip
\noindent\textbf{Acknowledgements.} During the work on this paper, the author was supported by the 
Russian Science Foundartion grant 16-11-10069, RFFI grant 18-01-0398 and the Simons foundation. The major part of the paper was written during the visit to PKU; the author expresses his deepest gratitude to this university for wonderful working conditions.

\subsection{The classical construction}\label{sec:secintro2}
In what follows we let $A$ be a Poisson algebra, i.e. a commutative unital algebra over a ground field $\Bbbk,\,char(\Bbbk)=0$ (usually $\Bbbk=\R$ or $\C$) with a bracket $\{,\}:A\otimes A\to A$, verifying the following set of relations:
\begin{description}
\item[{\rm(\textit{i})}] $\{,\}$ is bilinear over $\Bbbk$;
\item[{\rm(\textit{ii})}] $\{,\}$ is antisymmetric: $\{f,g\}=-\{g,f\}$;
\item[{\rm(\textit{iii})}] $\{,\}$ verifies the Leibniz rule:
\[
\{f,gh\}=\{f,g\}h+\{f,h\}g.
\]
\item[{\rm(\textit{iv})}] $\{,\}$ verifies the Jacobi identity:
\[
\{f,\{g,h\}\}+\{g,\{h,f\}\}+\{h,\{f,g\}\}=0.
\]
\end{description}
 Center $Z_{\pi}(A)$ of Poisson algebra $A$ is the subalgebra in $A$, spanned by the elements $x\in A$ such that $\{x,y\}=0$ for all $y\in A$; elements of $Z_\pi(A)\setminus\Bbbk$ are sometimes called \textit{Casimir elements} or simply \textit{Casimirs}.

If $A=C^\infty(M)$ is the algebra of smooth functions on a manifold, then conditions (\textit{i})-(\textit{iii}) mean that the bracket $\{,\}$ is determined by a bivector field $\pi\in\wedge^2TM$; in this case, as it is well-known (see e.g. \cite{Ciccoli}) the condition (\textit{iv}) is equivalent to the equation 
\begin{equation}\label{eq:Poi1}
[\pi,\pi]=0.
\end{equation}
Here $[,]$ is the Nijenhuis-Schouten bracket of polyvector fields. In this situation the bracket $\{f,g\}$ of any two functions is given by the formula:
\[
\{f,g\}=\pi(df,dg).
\]
One calls the bracket $\{,\}$ \textit{Poisson bracket} and the bivector $\pi$ verifying \eqref{eq:Poi1} \textit{Poisson bivector}. 

An important particular case is when $\pi$ has maximal rank; in this case its inverse differential $2$-form $\omega=\pi^{-1}$ satisfies the equation $d\omega=0$ and one calls it a \textit{symplectic structure} on $M$. Remark, that in symplectic case there are no Casimirs in $C^\infty(M)$, i.e. $Z_\pi(A)=\C$ or $\R$ in this case.

Let $\xi$  be a vector field on $M$; recall that one says that $\xi$ is \textit{Poisson} field, if the Lie derivative of $\pi$ with respect to it vanishes; this condition is equivalent to the equation
\begin{equation}\label{eq:Poivf}
\xi(\{f,g\})=\{\xi(f),g\}+\{f,\xi(g)\}.
\end{equation}
An important particular case of Poisson fields are the fields of the form $X_f=\pi^\sharp(df)$; here $\pi^\sharp$ is the ``index raising'' operator induced by $\pi$; in coordinates:
\[
\pi^\sharp(\alpha)^k=\pi^{kl}\alpha_l
\]
for any $1$-form $\alpha$. The fields $X_f$ are called \textit{Hamiltonian}, they are characterized by the equation $X_f(g)=-\{f,g\}$; the equality $\lc_{X_f}\pi=0$ now follows from the Jacobi identity.

Another important class of vector fields consists of \textit{Nijehuis} fields: \textit{we shall say, that a field $\xi$ is Nijenhuis, if the second Lie derivative with respect to $\xi$ kills the Poisson bivector $\pi$, i.e.
\begin{equation}\label{eq:defNj}
\lc_\xi^2\pi=0.
\end{equation}
}
A good example of Nijenhuis field is any constant (in linear coordinates) field on a vector space, in case $\pi$ is linear; for example, one can take $M=\g^*$ with the standard Poisson-Lie structure.

The purpose of considering Nijenhuis fields follows from the next observation (here $A(M)$ is the Poisson algebra of smooth functions on a manifold with respect to a bivector $\pi$):
\begin{prop}[The argument shift method]\label{prop:ashiftcl}
For any $f,\,g\in Z_\pi(A(M))$, any Nijenhuis vector field $\xi$ and any natural $k,\,l$ the following relation holds: $\{\lc_\xi^kf,\lc_\xi^lg\}=0$.
\end{prop}
Observe, that if $\xi$ is in fact Poisson field, then $\lc_\xi^kf\in Z_\pi(A(M))$ for all $k\in\mathbb N,\,f\in Z_\pi(A(M))$, so the statement holds trivially; however, in the case of a ``genuine'' Nijenhuis field, i.e. if $\lc_\xi\pi\neq0,\,\lc_\xi^2\pi=0$, the functions $\lc_\xi^kf$ are no more central. Let us now sketch the proof of this proposition:
\begin{proof}
We shall prove by induction in $N\ge0$ that $\lc^k_\xi\pi(d\lc_\xi^lf,d\lc_\xi^mg)=0$ for all $k+l+m=N$ and all $f,g\in Z_\pi(A(M))$ (this statement is a bit more than what the proposition needs, since $\pi(d\lc_\xi^kf,d\lc_\xi^lg)=\{\lc_\xi^kf,\lc_\xi^lg\}$). The base of this statement ($N=0$) is trivially true, since $f,g$ are central. 

Next, assuming that this statement holds for all small $N$, and differentiating the corresponding equations by $\xi$ for all $k,\,l$ and $m$ we obtain the following system of linear equations on the values of $\lc_\xi^k\pi(d\lc_\xi^lf,d\lc_\xi^mg)$:
\[
\lc^{k+1}_\xi\pi(d\lc_\xi^lf,d\lc_\xi^mg)+\lc^k_\xi\pi(d\lc_\xi^{l+1}f,d\lc_\xi^mg)+\lc^k_\xi\pi(d\lc_\xi^lf,d\lc_\xi^{m+1}g)=0,\ \forall k+l+m=N.
\]
Observe next that due to the conditions on $\xi$ and $\pi$ only $k=0$ and $k=1$ can appear, so this system is reduced to
\[
\begin{cases}
&\lc_\xi\pi(d\lc_\xi^lf,d\lc_\xi^mg)+\pi(d\lc_\xi^{l+1}f,d\lc_\xi^mg)+\pi(d\lc_\xi^lf,d\lc_\xi^{m+1}g)=0,\ l+m=N-1\\
&\lc_\xi\pi(d\lc_\xi^{l+1}f,d\lc_\xi^mg)+\lc_\xi\pi(d\lc_\xi^lf,d\lc_\xi^{m+1}g)=0,\ l+m=N-1.
\end{cases}
\]
We are going to show that this system under the conditions of our proposition can have only trivial solutions. To this end consider the second equation: varying $l$ from $0$ to $N-1$, we get inductively:
\[
\lc_\xi\pi(d\lc_\xi^{N-1},dg)=-\lc_\xi\pi(d\lc_\xi^{N-2}f,d\lc_\xi g)=\lc_\xi\pi(d\lc_\xi^{N-3}f,d\lc_\xi^2g)=\dots=(-1)^{N-1}\lc_\xi\pi(df,\lc_\xi^{N-1}g),
\]
i.e. $\lc_\xi\pi(d\lc_\xi^{N-k-1}f,d\lc_\xi^kg)=(-1)^k\lc_\xi\pi(d\lc_\xi^{N-1},dg)$.

Now we turn to the first series of equations: if $m=0$, it is reduced to
\[
\lc_\xi\pi(d\lc_\xi^{N-1}f,dg)+\pi(d\lc_\xi^{N-1}f,d\lc_\xi g)=0,
\]
since $g$ is central. So 
\[
\pi(d\lc_\xi^{N-1}f,d\lc_\xi g)=-\lc_\xi\pi(d\lc_\xi^{N-1}f,dg)=\lc_\xi\pi(d\lc_\xi^{N-2}f,d\lc_\xi g).
\]
Now from the the equation of the first series with $m=1$ we get:
\[
\lc_\xi\pi(d\lc_\xi^{N-2}f,d\lc_\xi g)+\pi(d\lc_\xi^{N-1}f,d\lc_\xi g)+\pi(d\lc_\xi^{N-2}f,d\lc_\xi^{2}g)=0,
\]
or 
\[
\pi(d\lc_\xi^{N-2}f,d\lc_\xi^{2}g)=-2\lc_\xi\pi(d\lc_\xi^{N-2}f,d\lc_\xi g)=2\lc_\xi\pi(d\lc_\xi^{N-1}f,dg)=2\lc_\xi\pi(d\lc_\xi^{N-3}f,d\lc_\xi^2g).
\]
We plug this into the equation for $m=2$ in the first series:
\[
\lc_\xi\pi(d\lc_\xi^{N-3}f,d\lc_\xi^2g)+\pi(d\lc_\xi^{N-2}f,d\lc_\xi^2g)+\pi(d\lc_\xi^{N-3}f,d\lc_\xi^{3}g)=0
\]
we get
\[
\pi(d\lc_\xi^{N-3}f,d\lc_\xi^{3}g)=-3\lc_\xi\pi(d\lc_\xi^{N-1},dg).
\]
By induction we get
\[
\pi(d\lc_\xi^{N-k}f,d\lc_\xi^{k}g)=(-1)^kk\lc_\xi\pi(d\lc_\xi^{N-1},dg).
\]
But when $k=N$ we see that $\pi(df,d\lc_\xi^Ng)=0$, since $f$ is central. So $\lc_\xi\pi(d\lc_\xi^{N-1},dg)=0$ and all the other expressions vanish automatically.
\end{proof}
\begin{rem}
In fact, this statement and its proof we gave here are applicable to binary operations of any kind $m:V\otimes V\to V$ on any vector space $V$ and any linear operator $\xi:V\to V$; we just put:
\[
\xi(m)(a,b)=\xi(m(a,b))-m(\xi(a),b)-m(a,\xi(b)),\ a,b\in V
\] 
for any two operators of this sort. Then \textit{if $\xi^2(m)=\xi(\xi(m))=0$, and 
\[
a,b\in Z_m(V)=\{x\in V\mid m(x,y)=0=m(y,x),\,\forall y\in V\}
\]
then for all $k,\,l\in\mathbb N_0$ we have $m(\xi^k(a),\xi^l(b))=0$}.
\end{rem}

As one knows from Kontsevich's theorem, for every Poisson structure $\pi$ on a manifold $M$ there exists an associative star-product $\star=\star_\pi$ on the space of formal power series with coefficiewnts in functions on $M$:
\begin{equation}\label{eq:star1}
f\star g=fg+\frac{\hbar}{2}\{f,g\}+\sum_{k=2}^\infty\hbar^kB_k(f,g),
\end{equation}
where $B_k(-,-)$ are suitable bidifferential operators; to make our notation more uniform we shall usually replace the term $\frac12\{f,g\}$ in this formula by $B_1(f,g)$. Below we shall give a brief description of Kontsevich's construction, proving the existence of this formula. It is known, that many properties of the original algebra of functions on $M$ can be easily transferred to the algebra $\ac(M)=(C^\infty(M)[[\hbar]],\star)$, called the deformation quantization of $M$. In particular, one can show that for every element in $Z_\pi(A(M))$ there will be a well-defined element $\tilde f\in Z(\ac(M))$ (where on the right we denote by $Z(\ac(M))$ the center of $\ac(M)$, i.e. the subalgebra of elements, commuting with every other element in $\ac(M)$). 

One can ask, \textit{if the argument shift method can be transferred to the quantized algebra as well}? The proof of proposition \ref{prop:ashiftcl} being purely algebraic and suitable for any type of binary operation, this conjecture seems quite plausible. However the answer to this question still remains unknown, although in some important particular cases it is known to be positive. For example, this is the case when $M=\g^*$ for a semisimple Lie algebra $\g$, but the construction of commutative families in that situation is quite complicated and is based on the subtle properties of affine Lie agebras and Yangians, i.e. on the study of infinite-dimensional Lie theory.

 The purpose of the remaining part of this paper is to describe a construction analogous to the argument shift method in which the notion of Nijenhuis vector field is replaced with a suitable construction from $L_\infty$ algebras. We hope, that this notion being somewhat less restrictive than the usual Nijenhuis condition \eqref{eq:defNj}, it will be possible to find instances of this structure in a wider set of examples (and hence to find large commutative subalgebras in quantized algebras).
 
\section{$L_\infty$ structures and morphisms}
\subsection{$L_\infty$-algebras}
We begin with the classical definition:
\begin{df}\label{def:linf}
One says that a graded space $V=\bigoplus_{n\ge0} V_n$ is given the structure of $L_\infty$-algebra, if its graded exterior algebra $\Lambda_V=\Lambda^*(V[1])$ is equipped with operator $D:\Lambda_V\to\Lambda_V$ of degree $1$ such that
\begin{description}
\item[{\rm(\textit{i})}] $D^2=0$;
\item[{\rm(\textit{ii})}] $D$ is differentiation with respect to the free cocommutative coalgebra structure $\Delta_V$ on $\Lambda_V$, i.e. $\Delta_V:\Lambda_V\to\Lambda_V\otimes\Lambda_V$ and
\[
\Delta_VD_V=(D_V\otimes 1+1\otimes D_V)\Delta_V.
\]
\end{description}
\end{df}
It follows from condition (\textit{ii}) that the map $D$ is determined by its ``Taylor series'' coefficients: $D=\{D_n\}_{n\ge 1}$ where $D_n:\Lambda^n(V[1])\to V[1]$ is a degree $1$ map (or, if we restore the original grading $D_n:\Lambda^n(V)\to V$ has degree $2-n$); now we can write the condition (\textit{i})  ``in coordinates'', i.e. in terms of $D_n$. For instance (we use the same symbols for homogeneous elements of $V$ and for their degrees, i.e. $a\in V_a,\,b\in V_b$ etc.):
\[
\begin{aligned}
&D_1(D_1(v))=0,\\
&D_1(D_2(a,b))-D_2(D_1(a),b)-(-1)^{a}D_2(a,D_1(b))=0.
\end{aligned}
\]
Thus, the operator $D_1$ plays the role of differential in $V$ and is usually denoted by $d$; similarly $D_2$ determines a skew symmetric binary operation $[,]$ on $V$ of degree $0$, commuting with $d$. now the relation for $D_3$ in this notation can be written as
\[
(dD_3)(a,b,c)=[[a,b],c]+(-1)^{a(b+c)}[[b,c],a]+(-1)^{b(a+c)}[[c,a],b],
\]
where $dD_3$ is the differential of a homogeneous map: $df=d\circ f-(-1)^ff\circ d$. In other words, $[,]$ verifies the (graded) Jacobi identity up to a homotopy. All the other relations are generalized Jacobi identities; they have the form:
\[
(dD_{n})(a_1,\dots,a_{n})=\sum_{i=2}^n\sum_{\sigma\in UnSh(i,n-i)}(-1)^{\epsilon(\sigma,a_1,\dots,a_n)}D_{n-i+1}(D_i(a_{\sigma(1)},\dots,a_{\sigma(i)}),a_{\sigma(i+1)},\dots,a_{\sigma(n)}).
\]
Here $UnSh(i,j)$ is the set of all \textit{unshuffles} of size $(i,j)$, i.e. of all permutations, whose inverse maps are $(i,j)$-shuffles (recall, that $\sigma\in S_{i+j}$ is called $(i,j)$-shuffle, if $\sigma(1)<\sigma(2)<\dots<\sigma(i),\,\sigma(i+1)<\sigma(i+2)<\dots<\sigma(i+j)$) and sign $\epsilon(\sigma,a_1,\dots,a_n)$ is determined by the Koszul rules (below we shall usually abbreviate this notation just to $\epsilon$ or $\epsilon(\sigma)$).

From these formulas it is easy to see that the usual DG Lie algebra structure on a space $\mathfrak g$ determines a particular kind of $L_\infty$-structure: we just put $D_1=d,\,D_2=[,]$ (the usual differential and the Lie bracket on $\mathfrak g$) and put $D_3=D_4=\dots=0$. In fact, $\Lambda_\mathfrak g$ with structure map $D$ in this case is just the Chevalley-Eilenberg complex with its standard differential.

\subsection{Morphisms and homotopies in $L_\infty$ category}
Given two $L_\infty$-algebras $V,\,W$ one can ask, what is the proper notion of morphisms between them. This question is answered by the following definition:
\begin{df}\label{def:linfmo}
A degree $0$ map $F:\Lambda_V\to\Lambda_W$ for any two $L_\infty$-algebras $V,\,W$ is called an $L_\infty$-morphism, if
\begin{description}
\item[{\rm(\textit{i})}] $D_W\circ F=F\circ D_V$;
\item[{\rm(\textit{ii})}] $F$ is a homomorphism with respect to the free cocommutative coalgebra structures on $\Lambda_V,\,\Lambda_W$.
\end{description} 
\end{df}
As before this condition can be expressed in terms of ``Taylor coefficients'': we represent $F$ as a collection of maps $F_n:\Lambda^n(V)\to W$ of degree $1-n$, then the condition (\textit{ii}) from definition \ref{def:linfmo} turns into the following equalities:
\[
\begin{aligned}
&F_1d_V=d_WF_1,\\
&F_1([a,b])-[F_1(a),F_1(b)]=d_W(F_2(a,b))+F_2(d_Va,b)+(-1)^aF_2(a,d_Vb),
\end{aligned}
\]
and so on; for arbitrary $n$ this turns into the following somewhat cumbersome equation
\begin{equation}\label{eq:Linfmor}
\begin{aligned}
&\sum_{i=1}^n\sum_\sigma (-1)^{\epsilon_1}F_{n-i+1}(D^V_i(a_{\sigma(1)},\dots,a_{\sigma(i)}),a_{\sigma(i+1)},\dots,a_{\sigma(n)})\\
                     &=\sum_{k=1}^n\frac{1}{k!}\sum_{i_1+\dots+i_k=n}\sum_{\tau}(-1)^{\epsilon_2}D^W_k(F_{i_1}(a_{\tau(1)},\dots,a_{\tau(i_1)}),\dots,F_{i_k}(a_{\tau(i_1+\dots+i_{k-1}+1)},\dots,a_{\tau(n)})).
\end{aligned}
\end{equation}
Here, as before, the second sum on the left is taken over all $(i,n-i)$ unshuffles, while the last sum on the right is over $(i_1,i_2,,\dots,i_k)$ unshuffles and the signs $\epsilon_1,\,\epsilon_2$ are chosen in accordance with the Koszul rules. Since the morphisms of $L_\infty$-algebras are determined by these ``Taylor coefficients'', below we shall sometimes write $F=\{F_n\}:V\to W$.

As before, the usual homomorphism of DG Lie algebras is an example of $L_\infty$-morphism: we set $F_k=0$ for $k\ge2$. However, even in case of Lie algebras one can consider ``genuine'' $L_\infty$ maps, i.e. maps which do not coincide with the usual Lie-algebraic morphisms. In fact, equation \eqref{eq:Linfmor} in this case turns (with previous notation) into:
\begin{equation}\label{eq:Linfmog}
\begin{aligned}
&(dF_n)(a_1,\dots,a_n)-\sum_{i<j}(-1)^{\epsilon(i,j)}F_{n-1}([a_i,a_j]_V,a_1,\dots,\widehat{a_i},\dots,\widehat{a_j},\dots,a_n)\\
&=\sum_{i+j=n}\frac12\sum_{\sigma}(-1)^{\epsilon(\sigma)}[F_i(a_{\sigma(1)},\dots,a_{\sigma(i)}),F_j(a_{\sigma(i+1)},\dots,a_{\sigma(n)})]_W.
\end{aligned}
\end{equation}
Finally, we give the following definition of homotopy between two $L_\infty$-morphisms:
\begin{df}\label{def:linfho}
Two $L_\infty$-maps $F,\,G:\Lambda_V\to\Lambda_W$ are called homotopic, if there exist a map $H:\Lambda_V\to\Lambda_W$ of degree $-1$, such that
\begin{description}
\item[{\rm(\textit{i})}] $F-G=D_WH+HD_V$;
\item[{\rm(\textit{ii})}] $H$ is a derivation with respect to the free cocommutative coalgebra structures on $\Lambda_V,\,\Lambda_W$, i.e.
\[
\Delta_W H=(H\otimes 1+1\otimes H)\Delta_V.
\]
\end{description}
\end{df}
Once again, due to the condition (\textit{ii}) of this definition, the first condition can be expressed in terms of the ``Taylor coefficients'' of $H$, i.e. $H=\{H_n\}_{n\ge1}$, where $H_n:\Lambda^nV\to W$ is a map of degree $-n$. Then the equation from (\textit{i}) takes the form of the following sequence of equalities:
\[
\begin{aligned}
F_1(a)-G_1(a)&=d_WH_1(a)+H_1d_V(a),\\
F_2(a,b)-G_2(a,b)&=dH_2(a,b)+\left(H_1([a,b]_V)-[H_1(a),b]_W-(-1)^{ab}[H_1(b),a]_W\right)
\end{aligned}
\]
and so on. In particular the map $H_1$ is just a chain homotopy between $F$ and $G$. The general formula looks rather intimidating:
\begin{equation}\label{eq:Linfhom}
\begin{aligned}
F_n&(a_1,\dots,a_n)-G_n(a_1,\dots,a_n)\\
      &=\sum_{i+j=n}\sum_{\sigma\in UnSh(i,j)}(-1)^{\epsilon(\sigma)}\Bigl(D_{W,j+1}(H_i(a_{\sigma(1)},\dots,a_{\sigma(i)}),a_{\sigma(i+1)},\dots,a_{\sigma(n)})\\
      &\quad\qquad\qquad\qquad\qquad\qquad+H_{j+1}(D_{V,i}(a_{\sigma(1)},\dots,a_{\sigma(i)}),a_{\sigma(i+1)},\dots,a_{\sigma(n)})\Bigr)
\end{aligned}
\end{equation}
In what follows we shall deal with $L_\infty$-algebras, corresponding to DG Lie algebras; in this case equation \eqref{eq:Linfhom} will look as follows:
\begin{equation}\label{eq:Linfhodg}
\begin{aligned}
F_n&(a_1,\dots,a_n)-G_n(a_1,\dots,a_n)\\
      &=dH_n(a_1,\dots,a_n)+\sum_{i=1}^n(-1)^{\epsilon(i)}[H_{n-1}(a_1,\dots,\widehat{a_i},\dots,a_n),a_i]_W\\
    &\qquad\qquad\qquad\qquad+\sum_{1\le i<j\le n}(-1)^{\epsilon(i,j)}H_{n-1}([a_i,a_j]_V,a_1,\dots,\widehat{a_i},\dots,\widehat{a_j},\dots,a_n).
\end{aligned}
\end{equation}
One of the main advantages of $L_\infty$-algebras and morphisms is the fact that in this context homotopy equivalence is equivalent to the equivalence in $L_\infty$ sense. 

More accurately, recall that \textit{quasi-isomorphism} of DG Lie algebras is a homomorphism $f:\mathfrak g\to\mathfrak h$, which induces isomorphism on the level of homology (observe that it is enough to speak about the isomorphism on the level of vector spaces in this case). It turns out that not every quasi-isomorphism of DG Lie algebras $f$ has homotopy inverse, i.e. not for every $f$ one can find a homomorphism $g:\mathfrak h\to\mathfrak g$, whose composditions with $f$ are homotopic to identity. But this is not so in the wider category of $L_\infty$-algebras.

Namely, one calls an $L_\infty$-morphism $F:\Lambda_V\to \Lambda_W$ \textit{quasi-isomorphism} if its first ``Taylor coefficient'' $F_1:V\to W$ induces isomorphism in cohomology. Then the following statement is true (see \cite{Kon97}):
\begin{prop}\label{prop:Linfqua}
Every quasi-isomorphism $F:\Lambda_V\to\Lambda_W$ of $L_\infty$-algebras is homotopy-invertible, i.e. there exists an $L_\infty$-morphism $G:\Lambda_W\to\Lambda_V$ such that $F\circ G$ and $G\circ F$ are homotopic to identity.
\end{prop}
In particular, by the virtue of this statement and due to the observations made in the previous paragraph every quasi-isomorphism of DG Lie algebras obtains its inverse in the framework of $L_\infty$ category.

\subsection{Quasi-isomorphisms and $\star$-products}\label{sec:MCel}
The role of the $L_\infty$ algebras and morphisms in deformation theory is based on the following observations: let a star-product \eqref{eq:star1} be given. Then the formal sum $B=\sum_{n\ge1}\hbar^n B_n$ determines a $2$-cochain in the ($\hbar$-linear) Hochschild complex of $A[[\hbar]]$; the associativity condition then is transcribed as the formula:
\begin{equation}\label{eq:MC1}
\delta(B)+\frac12[B,B]=0
\end{equation}
where $\delta$ is the Hochschild cohomology differential and $[,]$ is the Gerstenhaber bracket, which introduces the DG Lie algebra structure on the Hochschild cohomology complex with shifted dimension (in fact, Gerstenhaber bracket is a morphism of degree $-1$, so in order to turn the Hochschild complex with $[,]$ into a DG Lie algebra we must shift all dimensions by $1$).

The equation \eqref{eq:MC1} exists in the case of an arbitrary DG Lie algebra $\mathfrak g$: just replace $\delta$ with proper differential in $\mathfrak g$ and use the given Lie bracket. This equation is called \textit{Maurer-Cartan equation} and the elements $\Pi$ of degree $1$ in $\mathfrak g$ are called \textit{Maurer-Cartan elements}. It turns out that the Maurer-Cartan elements behave very well with respect to the $L_\infty$-morphisms between the DG Lie algebras. Namely:
\begin{prop}\label{prop:MC1}
Let $\Pi$ be a Maurer-Cartan element in $\mathfrak g$ and let $F:\Lambda_\mathfrak g\to \Lambda_\mathfrak h$ be an $L_\infty$-morphism between these algebras. Than the following formula (assuming the convergence of the series in the right) determines a Maurer-Cartan element in $\mathfrak h$:
\begin{equation}\label{eq:MC2}
F(\Pi)=\sum_{n\ge1}\frac{1}{n!}F_n(\Pi,\dots,\Pi).
\end{equation}
\end{prop}
The proof of this statement is by direct computation with the help of equation \eqref{eq:Linfmog} and we omit it.
\begin{rem}
In fact, the notion of Maurer-Cartan element can be defined for arbitrary $L_\infty$ algebra: it is such an element $\Pi\in V$, that $D_V(\exp(\Pi))=0$, where 
\[
\exp(\Pi)=\sum_{n\ge0}\frac{1}{n!}\underbrace{\Pi\wedge\dots\wedge\Pi}_{n\ \mbox{times}}.
\]
Observe that in a generic case $\exp(\Pi)\in\widehat\Lambda_V$, i.e. it is not an element of the exterior algebra, but only an element in its suitable completion. If we rewrite the equation $D_V(\exp(\Pi))=0$ in ``coordinate'' terms, we obtain the following equation:
\[
d\Pi+\sum_{n\ge2}\frac{1}{n!}D_n(\Pi,\dots,\Pi)=0.
\]
If $V=\mathfrak g$ is a DG Lie algebra, then this equation is reduced to \eqref{eq:MC1}, since $D_n=0\,n\ge3$. It turns out that in this context the statement of proposition \ref{prop:MC1} remains true, i.e. the formula \eqref{eq:MC2} determines a Maurer-Cartan element in $W$. The proof in this case is even easier: just observe that $F(\exp(\Pi))=\exp(F(\Pi))$ under the assumption of general convergence of all series.
\end{rem}
The observation, made in proposition \ref{prop:MC1} helps to relate the $L_\infty$ theory with the deformation problem: first of all, if $A=C^\infty(M)$ for some manifold $M$, one can consider a subcomplex of the Hochschild cohomology complex of $A[[\hbar]]$, spanned by the $\hbar$-linear \textit{local} cochains, i.e. by cochains, determined by polydifferential operators on $M$. Clearly, this subcomplex is closed with respect to the Gerstenhaber brackets, so we have a DG Lie algebra $\mathfrak g^\cdot=\mathcal D_{poly}^\cdot[1](M)[[\hbar]]$ of polydifferential operators on $M$ with Hochschild differential. It is our purpose to find in this DG Lie algebra a Maurer-Cartan element $B\in\mathfrak g^1=\mathcal D^1_{poly}[1](M)[[\hbar]]$, such that $B(f,g)=\frac{\hbar}{2}\{f,g\}+o(\hbar)$.

To this end we observe that the cohomology of local Hochschild complex $\mathcal D_{poly}^\cdot(M)$ on $M$ is isomorphic to the space of polyvector fields $\mathcal T_{poly}^\cdot(M)$ on $M$; the isomorphism is induced by the \textit{Hoschild-Kostant-Rosenberg map} $\chi:\mathcal T_{poly}^\cdot(M)\to\mathcal D_{poly}^\cdot(M)$, determined by the formula:
\[
\chi(\Psi^p)(f_1,\dots,f_p)=\langle df_1\wedge\dots\wedge df_p,\Psi\rangle,
\]
where $\langle,\rangle$ denotes the natural pairing between differential forms and polyvector fields. One can show that this map in fact on the level of cohomology induces not only an isomorphism of vector spaces, but also isomorphism of graded Lie algebras, if we use Gerstenhaber bracket to induce the Lie algebra structure on Hochschild cohomology (in effect it even induces the isomorphism of graded Poisson algebras, if we consider the wedge product on polyvectors and the standard product of Hochschild cochains). However, the map $\chi$ is not a homomorphism of DG Lie algebras (although it commutes with the differentials, if we allow zero differential on $\mathcal T_{poly}^\cdot(M)$). The same remarks of course are true with respect to the $\hbar$-linear complexes: one can define the map $\chi_\hbar:\mathcal T_{poly}^\cdot[1](M)[[\hbar]]\to \mathcal D_{poly}^\cdot[1](M)[[\hbar]]$, which will commute with differentials and induce the isomorphism of Lie algebras on the level of cohomology; however this map will not be a homomorphism of Lie algebras.

On the other hand, the bivector $\frac12\hbar\pi\in\mathcal T_{poly}^1[1](M)[[\hbar]]$ is a solution of the Maurer-Cartan equation if the differential in $\mathcal T_{poly}^1[1](M)[[\hbar]]$ is trivial, and hence as we have seen earlier, if the Hochschild-Kostant-Rosenberg map could be extended to a quasi-isomorphism of $L_\infty$ algebras, the formula \eqref{eq:MC2} applied to $\Pi=\frac12\hbar\pi$ would give us a Maurer-Cartan element in the complex $\mathcal D_{poly}^\cdot[1](M)[[\hbar]]$, beginning with $\chi_\hbar(\frac12\hbar\pi)$ as prescribed in the deformation problem.

The construction of such quasi-isomorphism for algebras of functions on $M=\R^n$ with arbitrary Poisson structure was eventually given by Kontsevich. We shall denote this quasi-isomorphism by $\uc=\{U_n\}:\mathcal T_{poly}^1[1](M)\to\mathcal D_{poly}^1[1](M)$, so that the $\star$-product associated with it is given by
\[
B=\sum_{n\ge 1}\frac{\hbar^n}{n!}U_n(\underbrace{\pi,\dots,\pi}_{n\ \mbox{times}}).
\]
Observe that the convergence is guaranteed in the context of formal power series. We are not going to discuss the details of this construction now; for our purposes it is important to know that \textit{when the $\star$-product is induced by an $L_\infty$-map, in particular, by Kontsevich's quasi-isomorphism, the same construction as above induces the map from $Z_\pi(A(M))$ to $Z(\ac(M))$: for any $f\in Z_\pi(A(M))$ we put
\[
\hat f=\sum_{n\ge0}\frac{\hbar}{n!}U_{n+1}(f,\underbrace{\pi,\dots,\pi}_{n\ \mbox{times}});
\]
an easy computation shows that when $f\in Z_\pi(A(M))$, the inner derivation of $\ac(M)$, induced by $\hat f$ is equal to $0$}. 
\begin{rem}
In effect, this claim follows from the next simple and rather well known observation: if $V$ is an $L_\infty$-algebra and $\Pi$ is a Maurer-Cartan element in $V$, then the formula
\[
D^\Pi_{V,n}(v_1,\dots,v_n)=\sum_{k\ge0}\frac{1}{k!}D_{n+k}(v_1,\dots,v_n,\underbrace{\Pi,\dots,\Pi}_{k\ \mbox{times}})
\]
(under the general assumption of convergence in all such formulas) determines a new $L_\infty$-structure in $V$; we shall denote $V$ equipped with this structure by $(V,D_V^\Pi)$. Similarly if $F=\{F_n\}:V\to W$ is an $L_\infty$-morphism, then the formula
\[
F^\Pi_n(v_1,\dots,v_n)=\sum_{k\ge0}\frac{1}{k!}F_{n+k}(v_1,\dots,v_n,\underbrace{\Pi,\dots,\Pi}_{k\ \mbox{times}})
\]
determines an $L_\infty$-morphism $F^\Pi=\{F_n^\Pi\}:(V,D_V^\Pi)\to (W,D_W^{F(\Pi)})$.

Now, the claim concerning the algebras' centres follows from the fact that in the context of Poisson algebra $f\in Z_\pi(A(M))$ if and only if $D^\pi_{A(M),1}(f)=0$ and $\hat f\in\ac(M)$ is in the center of $\ac(M)$ if and only if $D^B_{\ac(M),1}(\hat f)=0$.
\end{rem}

Thus for all elements in $Z_\pi(A(M))$ we have their counterparts in the center of $\ac(M)$. In the next section we shall discuss the analog of the Nijenhuis equation in the context of $L_\infty$-algebras and its meaning for constructing commutative subalgebras in $\ac(M)$.

\section{$L_\infty$-derivations and Nijenhuis property}

\subsection{Properties of $\mathcal T_{poly}^\cdot[1](M)$ and $\mathcal D_{poly}^\cdot[1](M)$ in low degerees}\label{sec:introLinfder}
In the remaining part of this note, we are going to deal only with DG Lie algebras, related with the deformation theory, i.e. $\mathcal T_{poly}^\cdot[1](M)$ and $\mathcal D_{poly}^\cdot[1](M)$ (although we shall often omit the shift from our notation). Thus it is worth beginning this section with a short list of properties, that we shall need.

First of all, the grading in both algebras begins with $-1$ (after shift), or from $0$; the differential vanishes on the lowest degree in both algebras and the lowest degree elements in both algebras commute. 

If we go further, we see that Hochschild-Kostant-Rosenberg map, although not a homomorphism of DG Lie algebras, intertwines the commutators in the lowest degrees (i.e $-1$ and $0$ in the shifted case): on the level $-1$ this is evident since on both sides commutator vanishes; for two vector fields $\xi,\,\eta\in\mathcal T_{poly}^0[1](M)$, their commutator is equal to the difference of their compositions:
\[
[\xi,\eta](f)=\xi(\eta(f))-\eta(\xi(f)),
\]
but the same is true for $\chi(\xi)=\xi$ and $\chi(\eta)=\eta$. Similarly, the commutator of $f\in\mathcal T_{poly}^{-1}[1](M)$ with $\xi$ is equal to $\xi(f)\in\mathcal T_{poly}^{-1}[1](M)$; and similarly Gerstenhaber bracket of $f$ and $\xi=\chi(\xi)$ gives the same result.

Also let us recall that the Poisson bracket of two functions can be written as the following combination of elements in $\mathcal T_{poly}^\cdot[1](M)$:
\[
\{f,g\}=[f,[\pi,g]]
\]
(where $[,]$ stand for the Schouten brackets). If we use the Lichnerowicz-Poisson differential $d_\pi(\psi)=[\pi,\psi]$ on $\mathcal T_{poly}^\cdot[1](M)$, we can rewrite this formula as $\{f,g\}=[f,d_\pi g]$. The skew-symmetry of this operation is then ensured by the Jacobi identity and the fact of commutativity of $\mathcal T_{poly}^{-1}[1](M)$, mentioned above.

Similarly, if $\Pi$ is a Maurer-Cartan element in $\mathcal D_{poly}^\cdot[1](M)$, then we have the following equation
\[
f\star g-g\star f=[f,[\Pi,g]],
\]
for all $f,g\in A(M)$, where $\star$ is the deformed multiplication, determined by $\Pi$ and $[,]$ is the Gerstenhaber bracket. In other words, the commutator in $\ac(M)$ is defined by the same formula as the Poisson algebra structure on $A(M)$, therefore we are going to denote this commutator by the same symbol $f\star g-g\star f=\{f,g\}$, when it can cause no ambiguity. As before, the skew-commutativity of this operation follows from Jacobi identity and the commutativity of the Gerstenhaber brackets on functions. Observe that in this context the Jacobi identity for both braces follows from the Maurer-Cartan equation and the fact that differentials vanish on degree $-1$ elements.

\subsection{$L_\infty$-derivations, Maurer-Cartan elements and Nijenhuis conditions}
We begin with the following definition, very similar to the definitions, given in previous section:
\begin{df}
Let $V$ be an $L_\infty$-algebra, in particular we can take $V=\g$, where $\g$ is a DG Lie algebra. A map $\mathscr X:\Lambda_V\to\Lambda_V$ is called $L_\infty$-derivation, if
\begin{description}
\item[{\rm(\textit{i})}] $\mathscr X$ is a coderivation of the free coalgebra;
\item[{\rm(\textit{ii})}] $\mathscr X$ commutes with the structure map $D$.
\end{description}
\end{df}
In particular, as before the map $\mathscr X$ is determined by its ``Taylor coefficients'' and one can write down the condition (\textit{ii}) in terms of these coefficients and the structure maps $D_{V,i}$. In what follows we shall only consider the $L_\infty$-derivations for DG Lie algebras, so let $\mathscr X=\{X_n\}$ be an $L_\infty$-derivation of a DGLa \g; then the maps $X_n:\wedge^n\g\to\g$ verify the equalities
\begin{equation}\label{eq:xxxder2}
\begin{aligned}
dX_{n+1}(a_0,\dots,a_n)&=\sum_{0\le i<j\le n}(-1)^{\epsilon_{i,j}}X_n([a_i,a_j],a_0,\dots,\widehat{a}_i,\dots,\widehat{a}_j,\dots,a_n)\\
                                        &+\sum_{i=0}^n(-1)^{\epsilon_i}[a_i,X_n(a_0,\dots,\widehat a_i,\dots,a_n)].
\end{aligned}
\end{equation}
Here as earlier $\epsilon_{i,j},\ \epsilon_i$ are the signs, determined by the permutations, which place the elements $a_i,a_j$ in front of the others and\ \ $\widehat{}$\ \ denotes the omission of an element.

Now for a MC element $\pi\in\g$ consider the element
\[
\mathscr X(\pi)=\sum_{n=1}^\infty \frac{1}{n!}X_n(\pi,\dots,\pi).
\]
As before we assume that all the formulas of this sort enjoy the convergence property. More generally, for arbitrary $a\in\g$ we put:
\[
\mathscr X_\pi(a)=\sum_{n=1}^\infty\frac{1}{(n-1)!}X_n(a,\pi,\dots,\pi).
\]
Then 
\begin{prop}
\label{prop:xscr1}
The element $\mathscr X(\pi)$ is closed with respect to the $\pi$-twisted differential in \g: $d_\pi(a)=da+[\pi,a]$, i.e. $d\mathscr X(\pi)=-[\pi,\mathscr X(\pi)]$. If in addition we suppose that $\mathscr X_\pi(\mathscr X(\pi))=0$, i.e.
\[
\sum_{p\ge 1}\frac{1}{(p-1)!}X_p(\pi,\dots,\pi,\mathscr X(\pi,\dots,\pi))=0,
\]
then it also verifies the equation $[\mathscr X(\pi),\mathscr X(\pi)]=0$. 
\end{prop}
\begin{proof}
We compute with the help of the equation \eqref{eq:xxxder2} (the signs do not appear, since the degree of $\pi$ is $0$ in $\wedge^*\g$):
\[
\begin{aligned}
d(\mathscr X(\pi))&=\sum_{n=1}^\infty \frac{1}{n!}\Bigl(dX_n(\pi,\dots,\pi)+\sum_{i=1}^nX_n(\pi,\dots,\underset{i}{d\pi},\dots,\pi)\Bigr)\\
                         &=\sum_{n=1}^\infty\frac{1}{n!}\Bigl(\sum_{1\le i<j\le n}X_{n-1}([\pi,\pi],\pi,\dots,\underset{i}{\widehat{\pi}},\dots,\underset{j}{\widehat{\pi}},\dots,\pi)\\
                         &\quad+\sum_{i=1}^n(-1)^{\epsilon_i}[\pi,X_{n-1}(\pi,\dots,\underset{i}{\widehat\pi},\dots,\pi)]+\sum_{i=1}^nX_n(\pi,\dots,\underset{i}{d\pi},\dots,\pi)\Bigr)\\
                         &=\sum_{n=1}^\infty\frac{1}{n!}\Bigl(\frac{n(n-1)}{2}X_{n-1}([\pi,\pi],\pi,\dots,\pi)+n[\pi,X_{n-1}(\pi,\dots,\pi)]\\
                         &\quad+nX_n(d\pi,\pi,\dots,\pi)\Bigr)\\
                         &=\sum_{n=1}^\infty\Bigl(\frac{1}{(n-1)!}X_n(d\pi,\pi,\dots,\pi)+\frac{1}{2(n-2)!}X_{n-1}([\pi,\pi],\pi,\dots,\pi)\\
                         &\quad+\frac{1}{(n-1)!}[\pi,X_{n-1}(\pi,\dots,\pi)]\Bigl)\\
                         &=\sum_{n=2}^\infty\frac{1}{(n-1)!}X_n(d\pi+\frac12[\pi,\pi],\pi,\dots,\pi)-\left[\pi,\sum_{n=1}^\infty\frac{1}{n!}X_n(\pi,\dots,\pi)\right]\\
                         &=-[\pi,\mathscr X(\pi)].
\end{aligned}
\]
Here we used the Maurer-Cartan equation $d\pi+\frac12[\pi,\pi]=0$. In a similar way, one can prove the equality
\begin{equation}
\label{eq:xpi1}
d\left(\mathscr X_\pi(a)\right)=[\pi,\mathscr X_\pi(a)]+[\mathscr X(\pi),a]+\mathscr X_\pi(d_\pi a).
\end{equation}
for arbitrary $a\in\g$. Now, from the latter formula, the equality $d_\pi\mathscr X(\pi)=0$ and the equation 
\[
\mathscr X_\pi(\mathscr X(\pi))=\sum_{p,q=1}^\infty \frac{1}{(p-1)!q!}X_{p}(\pi,\dots,\pi,X_q(\pi,\dots,\pi))=0,
\]
we see
\[
%\begin{aligned}
0=d(\mathscr X_\pi(\mathscr X(\pi)))=[\pi,\mathscr X_\pi(\mathscr X(\pi))]+[\mathscr X(\pi),\mathscr X(\pi)]+\mathscr X_\pi(d_\pi\mathscr X(\pi))=[\mathscr X(\pi),\mathscr X(\pi)].
%\end{aligned}
\]
\end{proof}
\begin{df}
We shall say, that the $L_\infty$-derivarion $\mathscr X$ of \g\ verifies weak $L_\infty$-Nijenhuis property with respect to a Maurer-Cartan element $\pi$ in $\g$, if $\mathscr X_\pi(\mathscr X(\pi))=0$.
\end{df}
\begin{rem}
The proposition \ref{prop:xscr1} means that $\pi+\mathscr X(\pi)$ is a Maurer-Cartan elementin \g :
\[
\begin{aligned}
d(\pi+\mathscr X(\pi))&+\frac12[\pi+\mathscr X(\pi),\pi+\mathscr X(\pi)]=\\
                                  &=d\pi+\frac12[\pi,\pi]+d\mathscr X(\pi)+\frac12([\pi,\mathscr X(\pi)]+[\mathscr X(\pi),\pi])+\frac12[\mathscr X(\pi),\mathscr X(\pi)]=0.
\end{aligned}
\]
Also observe that we can rewrite equation \eqref{eq:xpi1} as
\begin{equation}
\label{eq:xpi2}
(d_\pi\mathscr X_\pi)(a)=d_{\pi}(\mathscr X_\pi(a))-\mathscr X_\pi(d_\pi a)=[\mathscr X(\pi),a].
\end{equation}
\end{rem}

\subsection{$L_\infty$-Nijenhuis property and commutation relations}
We are going to find the relation between the map $\mathscr X_\pi$ and the brackets in \g. To this end we compute:
\[
\begin{aligned}
d(X_{p+2}&(x,y,\pi,\dots,\pi))=X_{p+2}(dx,y,\pi,\dots,\pi)+(-1)^{|x|}\Bigl(X_{p+2}(x,dy,\pi,\dots,\pi)\\
                 &+\sum_{i=3}^{p+2}(-1)^{|y|}X_{p+2}(x,y,\pi,\dots,\underset{i}{d\pi},\dots,\pi)\Bigr)+X_{p+1}([x,y],\pi,\dots,\pi)\\
                  &+\sum_{i=3}^{p+2}\left(X_{p+1}([x,\pi],y,\pi,\dots,\underset{i}{\widehat\pi},\dots,\pi)+(-1)^{|x|}X_{p+1}(x,[y,\pi],\pi,\dots,\underset{i}{\widehat\pi},\dots,\pi)\right)\\
                  &+\sum_{3\le i<j\le p+1}(-1)^{|x|+|y|}X_{p+1}(x,y,[\pi,\pi],\pi,\dots,\underset{i}{\widehat\pi}\dots,\underset{j}{\widehat\pi},\dots,\pi)\\
                  &+[x,X_{p+1}(y,\pi,\dots,\pi)]+(-1)^{|x||y|}[y,X_{p+1}(x,\pi,\dots,\pi)]\\
                  &+\sum_{i=3}^{p+1}[\pi,X_{p+1}(x,y,\pi,\dots,\underset{i}{\widehat\pi},\dots,\pi)].
\end{aligned}
\]
Here $|x|,\,|y|$ are the degrees of $x$ and $y$ in the exterior powers of \g. Multiplying this expression by $\frac{1}{p!}$ and summing up for $p=0,1,2,\dots$ we obtain the formula:
\[
d_\pi\mathscr X_\pi(x,y)=\mathscr X_\pi([x,y])-[x,\mathscr X_\pi(y)]-(-1)^{|x||y|}[y,\mathscr X_\pi(x)],
\]
where we use the notation $\mathscr X_\pi(x,y)=\sum_{p=0}^\infty\frac{1}{p!}X_{p+2}(x,y,\pi,\dots,\pi)$ and $d_\pi$ on the left hand side denotes the usual differential of a map. Let us now assume that $\g=\mathcal T_{poly}^\cdot[1](M)$ or $\g=\mathcal D_{poly}^\cdot[1](M)$ and consider $y=d_\pi g,\,x=f$, where $|f|=|g|=-1$. In this case $\mathscr X_\pi(x,y)=0$ (because of dimension restrictions) and we have:
\[
\mathscr X_\pi(d_\pi f,d_\pi g)=\mathscr X_\pi(\{f,g\})-[f,\mathscr X_\pi(d_{\pi}g)]+\{g,\mathscr X_\pi(f)\}
\]
where we use the observation that $[f,d_\pi g]=\{f,g\}$, the commutator of $f$ and $g$ with respect to $\pi$, see section \ref{sec:introLinfder}. Similarly, from \eqref{eq:xpi2} we see $[f,\mathscr X_\pi(d_{\pi}g)]=\{f,\mathscr X_\pi(g)\}+[f,[\mathscr X(\pi),g]]$ and we have:
\begin{equation}
\label{eq:xder1}
\mathscr X_\pi(d_\pi f,d_\pi g)=\mathscr X_\pi(\{f,g\})-\{\mathscr X_\pi(f),g\}-\{f,\mathscr X_\pi(g)\}-[f,[\mathscr X(\pi),g]].
\end{equation}
Another important example is $x=f,\,y=[\mathscr X(\pi),g]$: first we compute
\[
d_\pi\mathscr X_\pi(\mathscr X(\pi),g)=\mathscr X_\pi([\mathscr X(\pi),g])-[\mathscr X(\pi),\mathscr X_\pi(g)].
\]
Here we used the equation $d_\pi\mathscr X(\pi)=0$. Next, using this equation we compute
\begin{equation}
\label{eq:xder2}
\begin{aligned}[m]
d_\pi&\mathscr X_\pi(f,[\mathscr X(\pi),g])=\mathscr X_\pi([f,[\mathscr X(\pi),g]])-[f,\mathscr X_\pi([\mathscr X(\pi),g])]-[[\mathscr X(\pi),g],\mathscr X_\pi(f)]\\
        &=\mathscr X_\pi([f,[\mathscr X(\pi),g]])-[f,d_\pi\mathscr X_\pi(\mathscr X(\pi),g)]-[f,[\mathscr X(\pi),\mathscr X_\pi(g)]]-[\mathscr X_\pi(f),[\mathscr X(\pi),g]]
\end{aligned}
\end{equation}
Let now $f,\,g$ belong to the ``center of $\pi$-deformed product'', i.e. let $d_\pi f=d_\pi g=0$ (see section \ref{sec:introLinfder}). Then by \eqref{eq:xder1} we have:
\[
\mathscr X_\pi(\{f,g\})=\{\mathscr X_\pi(f),g\}+\{f,\mathscr X_\pi(g)\}+[f,[\mathscr X(\pi),g]],
\]
and so $\mathscr X_\pi(\{f,g\})=\{\mathscr X_\pi(f),g\}=\{f,\mathscr X_\pi(g)\}=[f,[\mathscr X(\pi),g]]=0$.
Next we apply \eqref{eq:xder1} to $f,\,\mathscr X_\pi(g)$ where $d_\pi f=0$:
\[
\mathscr X_\pi(\{f,\mathscr X_\pi(g)\})=\{\mathscr X_\pi(f),\mathscr X_\pi(g)\}+\{f,\mathscr X_\pi(\mathscr X_\pi(g))\}+[f,[\mathscr X(\pi),\mathscr X_\pi(g)]].
\]
So: 
\begin{equation}
\label{eq:xder3}
\{\mathscr X_\pi(f),\mathscr X_\pi(g)\}+[f,[\mathscr X(\pi),\mathscr X_\pi(g)]]=0.
\end{equation}
Similarly, taking the pair $\mathscr X_\pi(f),\,g$ with $d_\pi g=0$ we get:
\begin{equation}
\label{eq:xder4}
\{\mathscr X_\pi(f),\mathscr X_\pi(g)\}+[g,[\mathscr X(\pi),\mathscr X_\pi(f)]]=0
\end{equation}
(here we used the identity $[f,g]=0$ for all functions). Next, we apply equation \eqref{eq:xder2}: since $d_\pi f=d_\pi g=d_\pi\mathscr X(\pi)=0$ the left hand side of this equation vanishes. The first term on the right is equal to $0$ since $[f,[\mathscr X(\pi),g]]=0$, and the second term is equal to $\{f,\mathscr X_\pi(\mathscr X(\pi),g)\}=0$ since $f$ is in center. Thus this equation amounts to
\begin{equation}
\label{eq:xder5}
[f,[\mathscr X(\pi),\mathscr X_\pi(g)]]+[g,[\mathscr X(\pi),\mathscr X_\pi(f)]]=0.
\end{equation}
Now, summing up the equations \eqref{eq:xder3} and \eqref{eq:xder4} and subtracting \eqref{eq:xder5} we get:
\[
2\{\mathscr X_\pi(f),\mathscr X_\pi(g)\}=0.
\]
Thus, $\{\mathscr X_\pi(f),\mathscr X_\pi(g)\}=0$.

\section{Argument shift in deformed algebras}
\subsection{Strong Nijenhuis property}
We saw, that in the case when $\mathscr X$ verifies the weak Nijenhuis condition, the elements $\mathscr X_\pi(f),\,\mathscr X_\pi(g)$ commute, although they are not in general central: one sees from equation \eqref{eq:xpi2} that in this case
\[
d_\pi\mathscr X_\pi(f)=[\mathscr X(\pi),f],
\]
which need not vanish. However, weak Nijenhuis property is not enough to prove the commutativity of $\mathscr X_\pi^k(f),\,\mathscr X_\pi^l(g)$ for all $k$ and $l$. In order to prove this alongside the reasoning from the first section, we shall need a stronger condition. To this end we observe that the formal exponentiation of an $L_\infty$-derivation as a map from $\Lambda_V$ to itself gives an $L_\infty$-automorphism of $V$. Indeed, the commutation with $D_V$ follows from the definitions, and the fact that exponent of a coderivation is a homomorphism of coalgebras is trivial. As we have explained earlier in section \ref{sec:MCel}, any Maurer-Cartan element $\Pi$ in $V$ determines a deformation of the differential $D_{V,1}=d$ in $V$ and any $L_\infty$-morphism can be applied to Maurer-Cartan elements and can be extended to an $L_\infty$-morphism between the $L_\infty$-algebras with differentials, deformed by the Maurer-Cartan elements.

Summing up, we have the following a bit cumbersome proposition:
\begin{prop}
Let $V$ be an $L_\infty$-algebra, $\mathscr X=\{X_n\}_{n\ge1}:\Lambda_V\to\Lambda_V$ an $L_\infty$-derivation of $V$ and $\Pi\in V$ a Maurer-Cartan element. Let $\exp(\mathscr X)$ be the map with ``Taylor coefficients''
\[
\begin{aligned}
&\exp(\mathscr X)_n(a_1,\dots,a_n)=\\
&=\sum_k\frac{1}{k!}\sum_{i_1+i_2+\dots+i_p-p=n+1}\frac{1}{i_1!(i_2-1)!\dots(i_p-1)!}\\
 &\,\sum_{\sigma\in S_n}X_{i_p}(\dots(X_{i_2}(X_{i_1}(a_{\sigma(1)},\dots,a_{\sigma(i_1)}),a_{\sigma(i_1+1)},\dots,a_{\sigma(i_1+i_2-1)}),\dots),a_{\sigma(i_1+\dots+i_{p-1}+3-p)},\dots,a_{\sigma(n)}).
\end{aligned}
\]
Then $\exp(\mathscr X)$ is an $L_\infty$-morphism. Moreover, $\exp(\mathscr X)(\Pi)$ is a Maurer-Cartan element in $V$ and if $\mathscr X_\Pi$ is given by the formula
\begin{equation}\label{eq:pider}
\mathscr X_{\Pi,n}(a_1,\dots,a_n)=\sum_{p\ge0}\frac{1}{p!}X_{n+p}(a_1,\dots,a_n,\underbrace{\Pi,\dots,\Pi}_{p\ \mbox{times}})
\end{equation}
then $\mathscr X_\Pi$ is an $L_\infty$-derivative between $(V,D_{V,\Pi})$ (the $L_\infty$-algebra $V$ with $\Pi$-deformed differential $D_{V,\Pi}$) and $(V,D_{V,\exp(\mathscr X)(\Pi)})$, which means that the following equality holds:
\[
\mathscr X_\Pi\circ D_{V,\Pi}=D_{V,\exp(\mathscr X)(\Pi)}\circ\mathscr X_\Pi.
\]
In particular, $\exp(\mathscr X_\Pi)$ is an $L_\infty$-morphism between $(V,D_{V,\Pi})$ and $(V,D_{V,\exp(\mathscr X)(\Pi)})$.
\end{prop}
In the particular case $V=\g$ and $\Pi=\pi$ (the usual Maurer-Cartan element in DG Lie algebra) the map $\mathscr X_\pi=\{\mathscr X_{\pi,n}\}_{n\ge0}:\Lambda_\g\to\Lambda_\g$, with ``Taylor coefficients'' given by equation \eqref{eq:pider} will verify the following identities, similar to \eqref{eq:xxxder2}
\[
\begin{aligned}
(d\mathscr X_{\pi,n+1})(a_0,\dots,a_n)&=\sum_{0\le i<j\le n}(-1)^{\epsilon_{i,j}}\mathscr X_{\pi,n}([a_i,a_j],a_0,\dots,\widehat{a}_i,\dots,\widehat{a}_j,\dots,a_n)\\
                                        &+\sum_{i=0}^n(-1)^{\epsilon_i}[a_i,\mathscr X_{\pi,n}(a_0,\dots,\widehat a_i,\dots,a_n)].
\end{aligned}
\]
However, here $d\mathscr X_{\pi,n+1}$ stands for the commutator
\[
[d,\mathscr X_{\pi,n+1}]=d_{\exp(\mathscr X)(\pi)}\circ\mathscr X_{\pi,n+1}+(-1)^n\mathscr X_{\pi,n+1}\circ d_\pi.
\]
We shall call maps with this property \textit{twisted $L_\infty$-derivations}.
\begin{rem}\rm
Of course, all the statements of this section are still true for the maps $\exp(t\mathscr X),\,\exp(t\mathscr X_\pi)$ etc. We shall use this observation below.
\end{rem}

It is clear, that the maps we discussed above are closely related with this construction: $\mathscr X_\pi(a),\,\mathscr X_\pi(x,y)$ are just the $1$-st and the $2$-nd ``Taylor coefficients'' of this twisted $L_\infty$-derivation. This brings forth the following definition:
\begin{df}
We shall say, that $\mathscr X$ verifies the (strong) Nijenhuis condition with respect to $\pi$, if 
\[
\mathscr X_{\pi,n}(a_1,\dots,a_n)=0,\ \mbox{when $a_i=\mathscr X(\pi)$ for some $i=1,\dots,n$}.
\]
\end{df}
Of course, every strong Nijenhuis derivation verifies the weak Nijenhuis property, so all the statements from the previous section remain valid. Below we shall show that strong Nijenhuis property of an $L_\infty$-derivation of DG Lie algebras allows one perform the trick from the proposition \ref{prop:ashiftcl} and thus obtain commutative subalgebras in quantized algebras. One of the problems of the theory that we develop here is that so far we have no example of a ``genuine'' (strong) Nijenhuis derivations, which would not come from the considerations of the usual Nijenhuis fields on a Poisson manifold.

\subsection{Argument shift for strong Nijenhuis derivations}
Let now \g\ be a deformation the DG Lie algebra, i.e. $\g=\mathcal T_{poly}^{-1}[1](M)$ or $\g=\mathcal D_{poly}^{-1}[1](M)$. 
Now we claim that the following is true:
\begin{prop}\label{prop:shiftLinf1}
Let $\mathscr X$ be an $L_\infty$-derivation of a deformation DG Lie algebra \g, which verifies the strong Nijenhuis condition with respect to a Maurer-Cartan element $\pi$, then
\[
\mathscr X_{\pi,1}(\{x,y\})=\{\mathscr X_{\pi,1}(x),y\}+\{x,\mathscr X_{\pi,1}(y)\}+[x,[\mathscr X(\pi),y]]
\]
and 
\[
\mathscr X_{\pi,1}([x,[\mathscr X(\pi),y]])=[\mathscr X_{\pi,1}(x),[\mathscr X(\pi),y]]+[x,[\mathscr X(\pi),\mathscr X_{\pi,1}(y)]]
\]
for all $x,y\in\g^{-1}$ of the form $x=\mathscr X_{\pi,1}^k(f),\,y=\mathscr X_{\pi,1}^l(g),\, k,l=0,1,2,\dots$ with $f,\,g$ in the center of $\pi$, i.e. if $d_\pi f=0=d_\pi g$.
\end{prop}
\begin{proof}
First of all, observe that under the conditions we have $\exp(t\mathscr X)(\pi)=\pi+t\mathscr X(\pi)$. Now we have the following lemma
\begin{lemma}\label{lem:lemsh1}
Let $x_k$ be of the form $x_k=\mathscr X_{\pi,1}^k(f),\ d_\pi f=0$, as prescribed in proposition \ref{prop:shiftLinf1}. Then $d_\pi(x_k)=[\mathscr X(\pi),x_{k-1}]$.
\end{lemma}
\begin{proof}
The map $\exp(t\mathscr X_{\pi,1})=\sum_{k\ge0}\frac{t^k}{k!}\mathscr X_{\pi,1}^k$ is the $1$-st ``Taylor coefficient'' of $\exp(t\mathscr X_\pi),\,t\in\R$. Thus, it verifies the condition:
\[
\exp(t\mathscr X_{\pi,1})\circ d_\pi=d_{\exp(t\mathscr X)(\pi)}\circ\exp(t\mathscr X_{\pi,1}).
\] 
This is true for all $t\in\R$, so for every $k=1,2,\dots$ and every $f\in\g$ we have
\[
d_\pi\circ \mathscr X_{\pi,1}^k(f)+(-1)^{|f|}[\mathscr X(\pi),\mathscr X_{\pi,1}^{k-1}(f)]=-\mathscr X_{\pi,1}^k\circ d_\pi(f).
\]
Since $d_\pi f=0$, the statement follows.
\end{proof}
Now, using the equation \eqref{eq:xder1} we see that the first equation from proposition \ref{prop:shiftLinf1} would follow if we prove that $\mathscr X_{\pi,2}(d_\pi x,d_\pi y)=0$ for $x=x_k,\,y=y_l$ and for all $k,l=1,2,\dots$ as in the conditions of proposition \ref{prop:shiftLinf1}. To this end consider $\mathscr X_{\pi,3}(\mathscr X(\pi),x,d_\pi y)$; we may assume that $k,l>0$ since otherwise the statement is trivially true. Due to the dimensional restrictions
\[
\mathscr X_{\pi,3}(\mathscr X(\pi),x,d_\pi y)=0=\mathscr X_{\pi,2}(x,d_\pi y).
\]
On the other hand, since $\mathscr X_\pi$ is a twisted $L_\infty$-derivation and verifies the strong Nijenhuis property, $\mathscr X(\pi)$ is $d_\pi$-closed and we have:
\[
0=(d\mathscr X_{\pi,3})(\mathscr X(\pi),x_{k-1},d_\pi y)=\mathscr X_{\pi,2}([\mathscr X(\pi),x_{k-1}],d_\pi y)=\mathscr X_{\pi,2}(d_\pi x_k,d_\pi y).
\]
We used the statement of lemma \ref{lem:lemsh1} in the last equality.

Similarly, we have $(d\mathscr X_{\pi,2})(x,[\mathscr X(\pi),y])=0$ due to the dimension restrictions, result of lemma \ref{lem:lemsh1} and the equation we just proved; on the other hand
\[
(d\mathscr X_{\pi,2})(x,[\mathscr X(\pi),y])=-[x,\mathscr X_{\pi,1}([\mathscr X(\pi),y])]-[\mathscr X_{\pi,1}(x),[\mathscr X(\pi),y]]+\mathscr X_{\pi,1}([x_k,[\mathscr X(\pi),y]])
\]
Finally, consider $\mathscr X_{\pi,2}(\mathscr X(\pi),y)$: due to the strong Nijenhuis condition $(d\mathscr X_{\pi,2})(\mathscr X(\pi),y)=0$, on the other hand, due to the same condition we have
\[
(d\mathscr X_{\pi,2})(\mathscr X(\pi),y)=\mathscr X_{\pi,1}([\mathscr X(\pi),y])-[\mathscr X(\pi),\mathscr X_{\pi,1}(y)].
\]
Comparing the last two equalities we obtain the second formula from proposition \ref{prop:shiftLinf1}.
\end{proof}
Now the following statement is a direct consequence of proposition \ref{prop:shiftLinf1} and the method we used in the proof of the proposition \ref{prop:ashiftcl}:
\begin{theorem}\label{theo:om}
Let $\mathscr X$ be a strong Nijenhuis derivation of a deformation DG Lie algebra with respect to the Maurer-Cartan element $\pi$. Then for any $f,\,g\in\g^{-1}$, such that $d_\pi f=0=d_\pi g$, the elements $\mathscr X_{\pi,1}^k(f),\,\mathscr X_{\pi,1}^l(g),\,k,l=0,1,2,\dots$ commute with respect to $\pi$:
\[
\{\mathscr X_{\pi,1}^k(f),\mathscr X_{\pi,1}^l(g)\}=[\mathscr X_{\pi,1}^k(f),[\pi,\mathscr X_{\pi,1}^l(g)]]=0.
\] 
\end{theorem}

\section{Conclusions and remarks}
As we have just seen, the (strong) Nijenhuis property of an $L_\infty$ derivation allows one to reproduce in a word for word manner the proof of the proposition \ref{prop:ashiftcl}, thus yielding the commutative algebras in the deformation quantization of Poisson manifolds (in case the quantization is done according to the Kontsevich's recipe). Let us now briefly discuss the possible directions of future research in relation with this our construction and the other topics, arising from it.

The first and most acute problem of our result is that so far we know of no other examples of $L_\infty$-derivations, verifying the (strong) Nijenhuis property except for the classical ones, i.e. those which arise in the study of usual Poisson algebras, for instance the linear Poisson structures on Euclidean spaces (in particular on the dual spaces of Lie algebras). Finding such a nontrivial example (or disproving its existence) is the most important question to be addressed in any paper, dedicated to the elaboration of the methods, considered above. 

Next, the manner of our proof is not the most economical one. In fact, we just showed, that the map $\mathscr X_{\pi,1}$ in this case verifies the equations, similar to a Nijenhuis vector field $\xi$, see proposition \ref{prop:shiftLinf1}. Moreover, as we saw in section 3, one can get some intermediate result with a much a much less restrictive assumptions. Thus the question, which should also be addressed in a future investigation is whether the strong Nijenhuis property can in some way be relaxed. Some evident improvements in this side can be made right now; for instance, since our proof of proposition \ref{prop:shiftLinf1} only involved manipulations with the maps $\mathscr X_{k,\pi}$ for $k=1,2,3$, we could have freely removed the conditions, involving all other maps $\mathscr X_{k,\pi},\,k\ge4$ from our considerations. However, this is too small an improvement to make this point at present. And of course, this question is closely related with the previous one: one might suppose that relaxing the strong Nijenuis property would make the quest for the corresponding examples easier.

Another consideration, which can be helpful in the search of applications of this construction, is that in the case of semisimple Lie algebras, the quantum counterparts of the commutative subalgebras rendered by the argument shift method, are known; we imply at the results of Tarasov, Rybnikov, Molev and others (see \cite{Tarasov,Rybnikov,Molev2}). These algebras are certain commutative subalgebras inside the universal enveloping algebra, which coincide with the argument shift results modulo the terms, linear in deformation parameter; however, the methods in which they are obtained, are totally different from each other and from anything, resembling the argument shifting. Thus, one of the first questions, one could ask about these algebras, is whether there is any shifting construction, that would underlie these results. Another interesting observation is that all these results are about the semisimple Lie algebras, whereas the usual method is applicable to any Lie algebra, and even to any Poisson structure with any Nijenhuis field associated to it. The role of the semisimplicity assumption is very far from being clearly understood, as well as the degree to which it can be dispensed of.

Last, but not least, is the question about the homological meaning of the Nijenhuis property: as one knows, the major step towards the construction of Kontsevich's quantization is the observation, that solutions of Maurer-Cartan equations can be ``moved around'' by $L_\infty$-morphisms. Now,  consider the pair $(\pi,\xi)$, where $\pi$ is a solution of the Maurer-Cartan equation in a DG Lie algebra \g\ and $\xi$ is a Nijenhuis vector field for $\pi$ (regarded as a derivation of \g\ or more generally as a linear operator on \g). Let $F=\{F_n\}:\g\to\mathfrak h$ be an $L_\infty$-morphism of DG Lie algebras. What kind of structure can we induce on $\mathfrak h$ from $(\pi,\xi)$ with the help of $F$? In particular, if $F$ is an $L_\infty$-quasi-isomorphism, then can one use $F$ to obtain a Maurer-Cartan element $F(\pi)$ in $\mathfrak h$ with an $L_\infty$-derivation, verifying the weak or strong Nijenhuis condition, associated with $F(\pi)$?


\medskip
\begin{thebibliography}{99}
\bibitem{Manakov1976} S.V. Manakov. \textit{Note on the integration of Euler’s equations of the dynamics of an $n$-dimensional rigid body.} Functional Analysis and Its Applications, {\bf10}:328-329, 1976.
\bibitem{MiFo78} A.S. Mishchenko and A.T. Fomenko. \textit{Euler equations on finite-dimensional Lie groups}. Mathematics of the USSR-Izvestiya, {\bf 12}(2):371-389, 1978.
\bibitem{Bol1} A.V. Bolsinov. \textit{Compatible Poisson brackets on Lie algebras and the completeness of families of functions in involution.} Mathematics of the USSR-Izvestiya, {\bf 38}(1):69-90, 1992.
\bibitem{Bol2} A.V. Bolsinov, K.M. Zuev. \textit{A formal Frobenius theorem and argument shift.} Mathematical Notes, {\bf86}(1-2):10-18, 2009.
\bibitem{Sad} S.T. Sadetov. \textit{A proof of the Mishchenko-Fomenko conjecture.} Doklady Math., {\bf70}(1):634-638, 2004.
\bibitem{BolZh} A. Bolsinov, P. Zhang. \textit{Jordan-Kronecker invariants of finite-dimensional Lie algebras.} arXiv:1211.0579, 2012.
\bibitem{Izo} A. Izosimov. \textit{Generalized argument shift method and complete commutative subalgebras in polynomial Poisson algebras.} arXiv:1406.3777, 2014.
\bibitem{Kon97} M. Kontsevich. \textit{Deformation quantization of Poisson manifolds, I.} Lett. Math.Phys., {\bf66}(3):157-216, 2003.
\bibitem{FeiginFrenkel} B. Feigin, E. Frenkel. \textit{Affine Kac-Moody algebras at the critical level and Gelfand-Dikii algebras.} Int. Jour. Mod. Phys., {\bf A7} Supplement 1A:197-215, 1992.
\bibitem{Molev1} A. I. Molev. \textit{Yangians and their applications.} in: Handbook of algebra, vol. {\bf3}, North-Holland, Amsterdam, 2003, 907-959.
\bibitem{Molev3} A. Futorny, A. Molev. \textit{Quantization of the shift of argument subalgebras in type A.} Adv. Math. {\bf 285}:1358-1375, 2015.
\bibitem{Molev2} A. Molev, O. Yakimova. \textit{Quantisation and nilpotent limits of Mishchenko-Fomenko subalgebras.} Represent. Theory {\bf 23}:350-378, 2019.
\bibitem{Tarasov} A.A. Tarasov. \textit{On some commutative subalgebras of the universal enveloping algebra of the Lie algebra $\mathfrak{gl}(n,\mathbb C)$.} Math. Sbornik {\bf191}(9):115-122, 2000
\bibitem{Rybnikov} L.G. Rybnikov. \textit{The Argument Shift Method and the Gaudin Model.} Funktsional. Anal. i Prilozhen. {\bf 40}(3):30-43, 2006
\bibitem{Talalaev} D. Talalaev. \textit{Quantization of the Gaudin System.} arXiv:hep-th/0404153, 2004.
\bibitem{Ciccoli} N. Ciccoli, P. Witkowski. \textit{From Poisson to Quantum Geometry.} Warsaw, 2006
\end{thebibliography}
\end{document}